\documentclass[A4paper,11pt]{article} % Sets paper size
\usepackage[margin=1in]{geometry}    % Sets 1-inch margins
\usepackage{setspace}
\usepackage[utf8]{inputenc}
\usepackage[english]{babel}
\usepackage{amsmath,amssymb,amsthm}
\usepackage{mathtools}
\usepackage{xspace}
\usepackage{array}
\usepackage{authblk} % author adresses
\usepackage{doi}
\usepackage{orcidlink}
\usepackage{tabularx}
\usepackage{hyperref}
\hypersetup{
	pdftitle={Complexity of partitioned-items response problems:\\ matchings and perfect matchings},
	pdfauthor={Christoph Buchheim, Lowig Duer, Eva Ley, Maximilian Merkert, Komal Muluk},
	colorlinks,
	hidelinks,
	linkcolor=red,
	citecolor=green
}
\usepackage[capitalise]{cleveref} 
\usepackage{tcolorbox} % for highlited conventions

\usepackage{threeparttable}
\usepackage{makecell}
\usepackage{multirow}
\usepackage{booktabs}

\usepackage{tikz}
\usetikzlibrary{backgrounds,calc,decorations.pathreplacing, hobby, arrows, arrows.meta, decorations.pathreplacing, decorations.pathmorphing, shapes, shapes.symbols}
\tikzset{
	tiny vertex/.style={draw,black,circle,inner sep=0pt,minimum size=3pt,fill},
	box/.style={draw,gray,fill=none,rounded corners},
	curly/.style={-,decorate,decoration={snake,amplitude=.6mm,segment length=2mm,post length=0mm}}	
}
\colorlet{leadercolor}{blue}
\colorlet{followercolor}{red}

\usepackage[disable]{todonotes}
\theoremstyle{plain}
\newtheorem{theorem}{Theorem}[section]

\newtheorem{corollary}[theorem]{Corollary}
\newtheorem{proposition}[theorem]{Proposition}

\theoremstyle{remark}

\newtheorem{example}[theorem]{Example}
\theoremstyle{definition}

\makeatletter
\newcommand{\problemtitle}[1]{\def\PROBLEMTITLE{#1}}% 
\newcommand{\probleminput}[1]{\def\PROBLEMINPUT{#1}}% 
\newcommand{\problemquestion}[1]{\def\PROBLEMQUESTION{#1}}% 
\newcommand{\problemgoal}[1]{\def\PROBLEMGOAL{#1}}%
\newenvironment{problemenv}{
	\problemtitle{}\probleminput{}\problemquestion{}\problemgoal{}
}{
	\par\addvspace{0.5\baselineskip}
	\noindent\begin{tabularx}{\linewidth}{|@{\hspace{5pt}} l@{\hspace{5pt}} X l}
		\multicolumn{2}{@{}l}{\sc\PROBLEMTITLE} \\[1ex]
		\textbf{Given:} & \PROBLEMINPUT \\[1ex]
		\if\relax\detokenize\expandafter{\PROBLEMQUESTION}\relax
		\else
		\textbf{Question:} & \PROBLEMQUESTION \\
		\fi
		\if\relax\detokenize\expandafter{\PROBLEMGOAL}\relax
		\else
		\textbf{Goal:} & \PROBLEMGOAL \\
		\fi
	\end{tabularx}
	\par\addvspace{1.0\baselineskip}
}
\makeatother

\def\Z{\hbox{$\mathbb Z$}}
\def\Q{\hbox{$\mathbb Q$}}
\def\N{\hbox{$\mathbb N$}}
\newcommand{\abs}[1]{\lvert #1 \rvert}
\newcommand{\set}[1]{\left\{#1\right\}}
\DeclareMathOperator*{\argmax}{arg\,max}

\newcommand{\st}{\mathrm{\mathsf{s.t.}}}
\newcommand{\ie}{i.e.,\ }

\newcommand{\eg}{e.g.,\ }

\newcommand{\wrt}{with respect to\xspace}

\renewcommand{\epsilon}{\varepsilon}
\newcommand{\np}{\ensuremath{\mathsf{NP}}\xspace}
\newcommand{\p}{\ensuremath{\mathsf{P}}\xspace}

\newcommand{\fpt}{\ensuremath{\mathsf{FPT}}\xspace}

\newcommand{\problem}{\ensuremath{\mathcal{P}}\xspace}
\newcommand{\instance}{\ensuremath{\mathcal{I}}\xspace}

\newcommand{\rp}[1]{\textsc{RP(#1)}\xspace} 
\newcommand{\bil}[1]{\textsc{B(#1)}\xspace} 
\newcommand{\maxM}{\textsc{Max-M}\xspace}
\newcommand{\perfectM}{\textsc{Perfect-M}\xspace} 
\newcommand{\perfectBM}{\textsc{Perfect-BM}\xspace}
\newcommand{\maxBM}{\textsc{Max-BM}\xspace}

\newcommand{\BCBS}{\textsc{Balanced Complete Bipartite Subgraph}\xspace}
\newcommand{\bcbs}{{BCBS}\xspace}

\newcommand{\LL}{X}
\newcommand{\FF}{Y}
\newcommand{\feas}{{\cal F}}
\newcommand{\target}{{\cal Y}}
\newcommand{\bigM}{K}
\newcommand{\numset}[1]{[#1]}

\title{Complexity of partitioned-items response problems:\\ matchings and perfect matchings}
\author[1]{Christoph~Buchheim\orcidlink{0000-0001-9974-404X}}
\author[1]{Lowig~Duer\orcidlink{0009-0002-1990-4029}}
\author[2]{Eva~Ley\orcidlink{0009-0000-7927-4757}}
\author[2]{Maximilian~Merkert\orcidlink{0000-0002-7838-445X}}
\author[1]{Komal~Muluk\orcidlink{0009-0002-3915-7198}}

\affil[1]{
	Department of Mathematics, TU Dortmund University, Germany
	}
\affil[2]{
	Institute for Mathematical Optimization, TU Braunschweig, Germany
	}

\date{}

\begin{document}
\singlespacing
\maketitle

\begin{abstract}	
We consider bilevel optimization problems in which leader and follower jointly construct a feasible solution for an underlying combinatorial optimization problem. 
Response problems ask whether the leader can encourage---or, in the pessimistic setting, enforce---a reaction of the follower that includes a set of mandatory items while excluding 
a set of forbidden items.
Our investigation focuses on 
tractability results for various cases which emerge from different combinations of the total number of mandatory, forbidden, and neutral items.
After providing some results for response problems that hold for any underlying 
combinatorial optimization problem, we examine response problems over the maximum-weight matching problem and the minimum-weight perfect matching problem as illustrative and surprisingly varied examples.
Among other results, we show that the response problem is hard for even a single given mandatory or forbidden edge.
On the other hand, it is fixed-parameter tractable \wrt the total number of non-mandatory edges.
If, however, each follower's edge is either mandatory or forbidden,
the response problem for the perfect matching problem is solvable in polynomial time while it remains \np-hard for the maximum-weight matching problem.

\end{abstract}

%%%%%%%%%%%%%%%%%%%%%%%%%%%%%%%%%%%%%%%%%%%%%%%%%%%%%%%%

\thispagestyle{empty}
\clearpage
\setcounter{page}{1}

\section{Introduction}
\label{Sec:Intro}

In bilevel optimization, two players solve a nested optimization problem in a hierarchical way: The leader takes a decision first, which influences the parameters of the follower's optimization problem. The optimal solution of the latter in turn has an influence on the leader's objective value and potentially also the feasibility of the leader's decision. The leader thus has to anticipate the follower's reaction when taking her decision, which often renders bilevel optimization problems significantly harder than their single-level counterparts. For instance, bilevel linear programs are known to be \np-complete~\cite{Jeroslow1985,Buchheim2023BilevelInNP}. For a general discussion of bilevel linear or mixed-integer optimization, including solution approaches, we refer to~\cite{Dempe2020} or the very recent textbook~\cite{Beck26}.

In the context of combinatorial optimization, partitioned-items problems are a well-studied class of bilevel problems. 
Here, the ground set of the underlying problem is partitioned into items controlled by the leader and items controlled by the follower.
First, the leader chooses a subset of her items.
Then, the follower's task is to extend it to a feasible solution of the underlying problem by means of appending a subset of his items. 
The leader aims to optimize her objective function while the follower aims to optimize his own objective function, which is in general different from that of the leader.
For many underlying structures, this bilevel variant turns out to be one level harder in the polynomial hierarchy than their corresponding single-level counterparts. 
Among others, this has been shown for the assignment problem~\cite{Gassner2009,Fischer2021}, the spanning tree problem~\cite{Shi2019,BHH22_ComplexityBilevelSpanningTree}, the knapsack problem~\cite{CCLW14_ComplexityBilevelKnapsack}, and the independent set problem~\cite{Muluk2026}.
For the partitioned-items shortest path problem, the bilevel variant is~$\Sigma_2^\text{P}$-hard~\cite{HW25_ComplexityBilevelShortestPath}. However, this is caused by the implicit fixing of edges in the shortest path problem to be solved by the follower or, equivalently, by the appearance of negative weights, in which case the shortest simple path problem itself is \np-hard~\cite{GarJoh1979}.

Another branch of bilevel optimization problems asks whether the leader can incentivize a follower's reaction that belongs to a given set of desired responses by modifying parameter values, \eg objective coefficients.
For linear programs, this decision problem has been investigated in different cases~\cite{Buchheim2025}. 
In 
inverse optimization, the leader searches for a minimal perturbation of a given follower's objective. 
If there is a single desired follower's response, \ie if the response is completely determined, the inverse problem for polynomial-time solvable follower's problems remains solvable in polynomial time~\cite{AO01_InverseOptimization}.
Efficient solution methods have been proposed for several of these problems; see~\cite{Heuberger04_InverseCombinatorialOptimization} and references therein.
In contrast, if the leader is only interested in certain aspects of the follower's reaction but indifferent about the rest, the partial 
inverse problems can be \np-hard even if the follower's problem is solvable in polynomial time; see~\cite{LM25_PICPsWplus} and references therein.

\emph{Response problems} are a combination of the two previously mentioned classes of bilevel problems. 
The ground set of a combinatorial optimization problem is first partitioned into items controlled by either the leader or the follower, then the follower's items are further partitioned into items that the leader wants to enforce (\emph{mandatory items}), wants to prevent (\emph{forbidden items}), or is indifferent about (\emph{neutral items}).
If there are no neutral items, there is exactly one desired follower's response;
we call this setting a \emph{complete} response problem in contrast to a \emph{partial} response problem where several of the follower's responses suit the leader.
For the spanning tree problem, the complete response problem is \np-hard even if there is only a single mandatory edge~\cite{BHH22_ComplexityBilevelSpanningTree}.

In this paper, we study tractability of some response problems depending on the total number of mandatory, forbidden, and neutral items given in the instance, establishing fixed-parameter tractability in some \np-hard cases and providing an overview of the arising complexity cases.
We first formally define the response problem, show general results which hold for arbitrary underlying problems, and then focus on the matching response problem. 
The latter turns out to produce a surprisingly rich complexity landscape depending on the restrictions we make. 
The response problem for maximum-weight matching is \np-hard in general, even with only one mandatory edge and no forbidden edges or vice versa. 
\np-hardness also persists when all follower's edges are forbidden. 
On the other hand, the problem turns tractable if all follower's edges are mandatory. 
More generally, we show fixed-parameter tractability in the total number of non-mandatory follower's edges. 
When considering perfect matchings, the results are mostly similar, with an important exception: If all follower's edges are forbidden, the response problem is \np-hard for maximum matchings, but tractable for perfect matchings.
All our results, in particular the negative ones, are also valid for bipartite graphs. 

The remainder of this paper is organized as follows. We give a formal definition of the response problem in Section~\ref{Sec:GeneralSetting}, together with results that are independent of the underlying combinatorial structure. In Section~\ref{Sec:Matching}, we investigate the complexity of the response problem for perfect and maximum-weight matchings, including fixed-parameter tractability results. Section~\ref{Sec:Conclusions} concludes.

\section{Response Problems}
\label{Sec:GeneralSetting}

In this study on response problems, we investigate the degree to which the leader can influence the follower's reaction in a partitioned-items bilevel problem.
As the basis of a  partitioned-items problem, we first define a general combinatorial optimization problem:
\begin{problemenv}
    \problemtitle{Combinatorial Optimization Problem: \problem}
    \probleminput{A finite ground set~$I$, a set of feasible solutions~$\feas\subseteq 2^I$, and a function~$c\colon I\to\Q$.}
    \problemgoal{Maximize~$c(X)\coloneqq\sum_{i\in X}c(i)$ over all~$X\in\feas$.}
\end{problemenv}
\noindent Note that the set of feasible solutions~$\feas$ does not need to be given as a list of elements but can be encoded implicitly depending on the specific problem structure.

\paragraph{Partitioned-items problems}
In partitioned-items bilevel problems, the ground set~$I$ of a combinatorial optimization problem instance $(I,\feas,c)$ of~\problem\ is partitioned into a set~$I_\ell\subseteq I$ of items controlled by the leader and the remaining items~$I_f\coloneqq I\setminus I_\ell$ controlled by the follower.
First, the leader chooses a subset of items from~$I_\ell$,
to this, the follower reacts by choosing a subset from~$I_f$; both subsets together must form a feasible solution of~\problem, \ie an element of~$\feas$.
Both decision makers optimize their own objective functions, denoted by $c\colon I \to \Q$ for the leader and $d\colon I_f \to \Q$ for the follower.
Note that the leader's objective value also depends on the follower's choice.

If there are several possibilities for the follower that are optimal \wrt his objective, these can result in different values for the leader's objective.
In the optimistic setting, the leader selects among these according to her objective.
Thus, the optimistic problem can formally be stated as follows:
\begin{problemenv}
    \problemtitle{Partitioned-items Problem: \bil{\problem}}
    \probleminput{An instance~$(I,\feas,c)$ of \problem, a partition~$I=I_\ell\ \dot\cup\ I_f$, and a function $d\colon I_f\to \Q$.}
    \problemgoal{Compute an optimal solution to
        \begin{equation*}
        \begin{aligned}
            \max_{\LL, \FF} \quad & c(\LL\cup\FF)\\ 
            \st \quad & \LL\subseteq I_\ell\\
            & \FF\in\argmax_{\FF'}~\{d(\FF')\mid \FF'\subseteq I_f,~\LL\cup \FF'\in\feas\}\;.\\[-2.5ex]
        \end{aligned}
        \end{equation*}}
\end{problemenv}

\pagebreak[2]

\noindent In contrast, in the pessimistic setting, an optimal follower's solution that is worst for the leader is chosen.
However, for partitioned-items problems, it suffices to consider one of the settings, as there exist polynomial transformations between them; see \cref{lem:trans-bil-opt-pes} in the appendix.

\paragraph{Response problems}
A \emph{response problem} asks whether there exists a leader's action that makes the follower choose a reaction containing some given mandatory items while avoiding some given forbidden items.
More precisely, the follower's item set $I_f$ is partitioned into \emph{mandatory items}~$I_1$, \emph{forbidden items}~$I_0$, and \emph{neutral items}~$I_*$.
In the optimistic setting, 
the problem is defined as: 
\begin{problemenv}
    \problemtitle{Response Problem: \rp{\problem}}
    \probleminput{An instance~$(I,\feas,c)$ of \problem, a partition~$I=I_\ell\ \dot\cup\ I_f$, a  function $d\colon I_f\to \Q$, and a partition~$I_f=I_0\ \dot\cup\ I_1\ \dot\cup\ I_*$.}
    \problemgoal{Decide whether there exists a leader's action $\LL\subseteq I_\ell$ such that there is an optimal follower's reaction $\FF$ to $\LL$ in \bil{\problem} that satisfies~$I_0\cap \FF=\emptyset$ and~$I_1\subseteq\FF$.}
\end{problemenv}

\noindent In other words, we want to decide whether the leader can make the follower choose a solution belonging to the \emph{target set} defined by
$\target\coloneqq \{\FF\subseteq I_f\mid I_0\cap\FF=\emptyset,~I_1\subseteq \FF\}$. 
Note that~$|\target|=2^{|I_*|}$, so that~$|\target|=1$ if and only if~$I_f=I_0\cup I_1$. 
We refer to the latter case as a \emph{complete response} problem, otherwise we speak about 
a \emph{partial response} problem.

\bigskip

The response problem is a special case of the decision version of the partitioned-items problem mentioned above:
Given an instance of the response problem \rp{\problem}, where \problem is a maximization problem and~$I=I_\ell\ \cup\ (I_0\ \cup\ I_1\ \cup\ I_*)$ is defined as above, we can create an equivalent instance of the partitioned-items problem \bil{\problem} by setting the leader's objective function as follows: 
\[
c(i)\coloneqq \begin{cases}\begin{array}{rl}
    1 & \text{ if } i\in I_1, \\
    -1 & \text{ if } i\in I_0, \\
    0 & \text{ if } i\in I_*\cup I_\ell\;.
    \end{array}
\end{cases}
\]
Then, the task is to verify whether the leader can achieve an objective value of at least $|I_1|$.

As for the partitioned-items problem, there is an optimistic and a pessimistic setting for the response problem.
In both settings, for a yes-instance, there exists a leader's action~$\LL$ to which the follower has at least one feasible response. 
In the optimistic setting, we additionally require that 
at least one of the optimal follower's reactions to~$\LL$ belongs to the target set~$\target$.
In the pessimistic setting, all possible optimal follower's reactions to~$\LL$ must belong to~$\target$.
The transformation between optimistic and pessimistic setting for the partitioned-items problems in \cref{lem:trans-bil-opt-pes} directly applies to response problems, as it only modifies the follower's objective.

\begin{corollary} \label[corollary]{cor:trans-rp-opt-pes}
    Given a combinatorial optimization problem~$\problem$, any optimistic (resp.\ pessimistic) instance of~$\rp{\problem}$ can be polynomially transformed into an equivalent pessimistic (resp.\ optimistic) instance of~$\rp{\problem}$ by only transforming the follower's objective function~$d$.
\end{corollary}

\noindent
As the optimistic and pessimistic setting of the response problem are equivalent when there is no restriction on the follower's objective, their computational complexity does not differ. 
Thus, from here on, we only study the optimistic setting of the response problem. 

Due to the response problem being a special case of the partitioned-items problem,
the computational complexity of the response problem can only rise by at most one level within the polynomial hierarchy compared to the underlying problem; see \cref{lem:complexity-bil}.

\begin{corollary}\label[corollary]{cor:complexity-rp}
    $\rp{\problem}\in\np^\problem$ for any combinatorial optimization problem~$\problem$.
\end{corollary}

\begin{example}\label[example]{ex:selection}
    To illustrate the response problem, let us consider \textsc{Selection} as the underlying problem~\problem. We are then given a set~$I$ of~$n$ items along with a weight function~$c\colon I \to\Q$ and some scalar~$p\in\N$, defining the feasible set~$\feas \coloneqq \{S\subseteq I \colon |S|=p\}$, 
    and the task in \problem is to select a subset~$S\subseteq I$ of exactly $p$ items that maximizes the weight~$c(S)$. 

    We claim that \rp{Selection} is solvable in $\mathcal{O}(n)$ time:
    For an instance \instance of \rp{Selection}, let $\pi$ be a non-increasing ordering of $I_f$ \wrt $d$, \ie let $d(i_{\pi(1)}) \geq \dots \geq d(i_{\pi(|I_f|)})$.
    For items with the same weight, let the mandatory items appear first, followed by neutral, and then followed by forbidden items.
    Let $i_{\pi(a)}$ be the last mandatory and~$i_{\pi(b)}$ the first forbidden item in the ordering.
    We claim that \instance is a yes-instance if and only if $p\geq a$, $p-|I_\ell|<b$, and $a<b$.

    Indeed, if $\instance$ is a yes-instance and $\LL\subseteq I_\ell$ is a corresponding leader's decision, the follower reacts by selecting the first $p-|\LL|$ items according to the ordering $\pi$, which then contain all mandatory and no forbidden items, so $a \leq p-|\LL| < b$.
    Since $0\leq |\LL|\leq |I_\ell|$, this implies $p\geq a$, $p-|I_\ell|<b$, and $a<b$.
    Conversely, let $p\geq a$, $p-|I_\ell|<b$, and $a<b$ hold.
    Let the leader select any $p-b+1\leq|I_\ell|$ items from $I_\ell$.
    Then the follower selects the first $b-1$ items from the ordering $\pi$.
    In particular, since $p \geq a$ and $a<b$, the follower selects all mandatory items and none of the forbidden items.
    Thus, $\instance$ is a yes-instance of the \rp{Selection} problem.

    Consequently, \rp{Selection} reduces to computing the positions~$a$ and~$b$. This is easy after sorting in~$\mathcal{O}(n\log n)$ time, but can also be done in linear time: To compute~$a$, it suffices to scan all mandatory items in any order, always updating the item~$i'$ with the smallest~$d(i')$, and then set~$a\coloneqq|I_1|+|\{i\in I_*\cup I_0\mid d(i)>d(i')\}|$. 
    The value~$b$ can be computed analogously. \qed
\end{example}

\Cref{ex:selection} shows that, depending on the underlying problem,~\rp{\problem} can also be~\problem-easy, so that \Cref{cor:complexity-rp} only provides an upper bound on the complexity of~\rp{\problem}.
In fact, solving~\rp{Selection} is not only possible in linear time but even surprisingly straightforward compared to the well-known linear-time \emph{median-of-medians} algorithm for \textsc{Selection}~\cite{BFPRT73_TimeBoundsSelection}.

Moreover, special cases of the response problem \rp{\problem} can belong to lower levels of the polynomial hierarchy than the bilevel problem \bil{\problem} over the same underlying combinatorial problem~\problem, or even 
than 
\problem itself, as the following example shows.

\begin{example}\label[example]{ex:knapsack}
    Consider the binary knapsack problem {\sc Knapsack}.
    We claim that \rp{Knapsack} with $I_f=I_1$ is in $\p$, while the general complete \rp{Knapsack} is in~$\p^\np$. Partial \rp{Knapsack} is even $\Sigma_2^\p$-hard; this follows from the proof of Theorem~3.1\,(b) in~\cite{CCLW14_ComplexityBilevelKnapsack}. 
    
    To show the first two claims, let a complete \rp{Knapsack} instance with weights~$a_i$,~$i\in I$, and a weight limit~$b$ be given and define~$b_1\coloneqq\sum_{i\in I_1}a_i$. 
    If the follower is given a weight limit of~$b_1$ and does not pack exactly~$I_1$, the instance is a no-instance. 
    Otherwise, there exists a~$b_{\max}\in (b_1,\infty]$ such that the follower packs exactly~$I_1$ for any weight limit in~$[b_1,b_{\max})$, while for any larger weight limit he will choose another solution.

    If all follower's items are mandatory, then $b_{\max}=\infty$.    
    A complete \rp{Knapsack} instance with~$I_f=I_1$ is a yes-instance if and only if the empty leader's choice is feasible.
    Thus, it suffices to check whether $b_1 \leq b$ and all follower's items have non-negative follower's values.

    For the general complete \rp{Knapsack}, we solve two knapsack problems.
    First, compute the optimal value $b_\ell$ of a knapsack problem on only the leader's items with weights and costs~$a_i$ and weight limit~$b-b_1$. 
    In other words, determine the minimum space possible to leave to the follower that is at least~$b_1$. 
    Next, let the follower solve his knapsack problem with the remaining capacity~$b-b_\ell$. 
    Then, the instance is a yes-instance if and only if the follower chooses exactly the mandatory items~$I_1$. 
    Note that this algorithm does not require to compute~$b_{\max}$ but only relies on its existence.
    \qed
\end{example}

In the following section, we present an in-depth complexity study of the response problem for maximum-weight and perfect matchings. 
We will see that the hardness strongly depends on the conditions posed on the sets~$I_0$, $I_1$, and~$I_*$. 
Along these lines, we first state two simple observations that are independent of the underlying problem~\problem.

\begin{proposition}\label[proposition]{pro:polytime:I_f=I_*}
    If~$I_f=I_*$, then \rp{\problem} is equivalent to deciding whether \problem admits any feasible solution. 
\end{proposition}
\begin{proof}
    Essentially, the leader does not care which follower's items are included.
    Thus, an instance of \rp{\problem} is a yes-instance if and only if its feasible set $\feas$ is non-empty.
\end{proof}
Moreover, the response problem is oracle-fixed-parameter tractable (oracle-\fpt) \wrt the number of leader's items:
If~$|I_\ell|\le k$ for some parameter $k$, the number of possible leader's solutions~$\LL \subseteq I_\ell$ is at most~$2^k$. 
A certificate $\LL$ can be found by enumeration based on \Cref{lem:complexity-bil}.
\begin{corollary}\label[corollary]{cor:ftp-rp}
    Given an oracle for~\problem, the problem \rp{\problem} is oracle-\fpt with respect to~$|I_\ell|$.
\end{corollary}
\section{Matching Response Problems}
\label{Sec:Matching}

We now focus on the maximum-weight matching (\maxM) and the minimum-weight perfect matching (\perfectM) problems, and study the complexity of their corresponding response problems, \rp{\maxM} and \rp{\perfectM}.
In both problems, we are given a graph~$G$ with edges partitioned into leader's and follower's edges as well as a follower's weight function~$d$.
Follower's edges are further partitioned into mandatory, forbidden, and neutral edges.
In \rp{\maxM}, the feasible set~$\feas$ consists of all matchings of~$G$, while it only contains perfect matchings, \ie matchings that cover all vertices of~$G$, in \rp{\perfectM}.
The response problems ask whether the leader can choose a matching among her edges such that the follower can extend it by choosing a matching including \emph{all} mandatory and \emph{no} forbidden edges.

Due to~\Cref{cor:complexity-rp}, both \rp{\maxM} and \rp{\perfectM} belong to \np.
However, their corresponding response problems may have different complexity, in contrast to \maxM\ and \perfectM\ being equivalent as single-level problems; see, e.g., \cite[Proposition~11.1]{KorteVygen_CombOpt}.
Throughout this section, all results concerning intractability are shown on bipartite graphs and all results concerning tractability are shown on general graphs, so they all hold for both settings.
We denote the matching problems on bipartite graphs by \textsc{\maxBM} and \textsc{\perfectBM}.
An overview of our results on response problems for matching problems can be found in~\Cref{tab:overview} in Section~\ref{Sec:Conclusions}.

\subsection{Complete Response Setting}

We first consider complete responses with different combinations of forbidden and mandatory sets.
To begin with, we consider the case where all the  follower's edges are mandatory.

\begin{proposition}\label[proposition]{prop:polytime:I_f=I_1}
   If~$I_f=I_1$, then \rp{\maxM} is polynomial-time solvable.
\end{proposition}
\begin{proof}
    An instance of \rp{\maxM} with~$I_f=I_1$ is a yes-instance if and only if~$I_f$ is a matching and all edges have non-negative weight \wrt $d$.
\end{proof}

While the complete \rp{\maxM} with only mandatory edges hence turns out to be solvable in polynomial time, in stark contrast to that, \rp{\maxM} with only forbidden follower's edges is already \np-hard.
Besides, it is hard even on the special case of bipartite graphs.
To show this hardness result,
we give a polynomial transformation from the following problem.

\begin{problemenv}
	\problemtitle{\BCBS (\bcbs)}
	\probleminput{A bipartite graph $H$ 
    and~$k\in\N$.}
	\problemgoal{Decide whether $H$ contains a complete bipartite subgraph with $k$ vertices on each side of the bipartition.}
\end{problemenv}
\noindent
\BCBS is \np-hard~\cite{GarJoh1979, Johnson1987}.

\begin{theorem}\label{thm:problembm-np-hard}
    \rp{\maxBM} is \np-complete even when~$I_f=I_0$.
\end{theorem}
\begin{proof}
    Given an instance $\instance = (H, k)$ of \bcbs, with~$H=(V_H,E_H)$ and bipartition~$V_H = V_H^-\ \dot\cup\ V_H^+$, 
	we construct an instance of \rp{\maxBM} as follows; see~\Cref{fig:response-max-bipartitematching} for an illustration:
    We first set~$n^-\coloneqq |V_H^-|-k$ and~$n^+\coloneqq |V_H^+|-k$ and define the bipartite graph~$G~=~(V^-\ \dot\cup\ V^+,E)$ by
    \[
    V^-\coloneqq V_H^-\cup\{v_1^-,\dots,v_{n^+}^-\},\quad V^+\coloneqq V_H^+\cup\{v_1^+,\dots,v_{n^-}^+\},
    \]
    and~$E=I_\ell\cup I_f$ with
    \begin{eqnarray*}
    I_\ell & \coloneqq  & \{(u,v_i^+)\mid u\in V_H^-,i\in\numset{n^-}\}\cup\{(v_i^-,u)\mid i\in\numset{n^+},u\in V_H^+\},\\
    I_f & \coloneqq  & \{(u,v)\mid u\in V^-_H, v\in V^+_H, (u,v)\not\in E_H\},\;
    \end{eqnarray*}
    where $\numset{n} \coloneqq \set{1, \dots, n}$ for $n \in \N$.
    Moreover, we define~$I_0\coloneqq I_f$, $I_1\coloneqq I_*\coloneqq \emptyset$, and $d(e)\coloneqq 1$ for all~$e\in I_f$.
    By construction, any matching~$M\subseteq I_\ell$ covers at most~$n^-$ vertices from~$V^-$ and at most~$n^+$ vertices from~$V^+$.
	Thus, at least~$k$ vertices from $V^-$ and at least~$k$ vertices from~$V^+$ are not covered by any matching~$M$. 
    
    For a yes-instance of \rp{\maxBM}, there exists a matching $M$ such that the vertices not covered by~$M$ do not share any edge.
    Otherwise, the follower would respond with a non-empty matching. 
    Then, the corresponding vertices in~$H$ form a complete bipartite subgraph, \ie the instance~$\instance$ of \bcbs is a yes-instance.
    
    Conversely, assume that $\instance$ is a yes-instance of \bcbs, \ie let $H$ contain a complete bipartite subgraph with~$k$ vertices on each side. 
    Then, the corresponding vertices do not share any edge in~$G$, and the leader can choose~$n^-+n^+$ of her edges that cover all remaining vertices of~$H$. 
    Consequently, the follower cannot extend this matching by any other edge, since the only remaining uncovered vertices form an independent set. Thus, the constructed instance of \rp{\maxBM} is a yes-instance.
\end{proof}
	
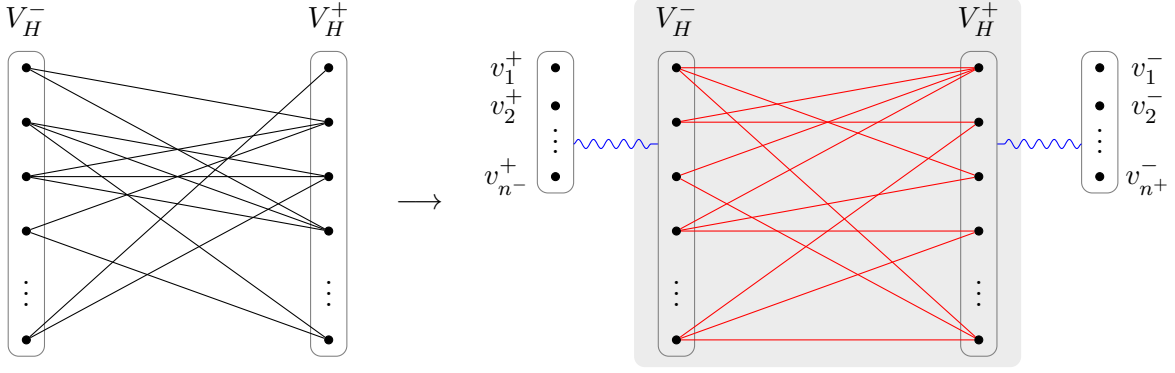
\begin{figure}[t]
	\begin{center}
	\begin{tikzpicture}[yscale=.9,scale=0.8]
        \def\dis{5}
        \begin{scope}
            \node at (0, 10.9) {$V_H^-$};
            \node[tiny vertex] (a1) at (0, 10) {};
            \node[tiny vertex] (a2) at (0, 9) {};
            \node[tiny vertex] (a3) at (0, 8) {};
            \node[tiny vertex] (a4) at (0, 7) {};
            \node at (0, 6) {$\vdots$};
            \node[tiny vertex] (an) at (0, 5) {};
            \draw[box] ($(a1)+(-.3, .3)$) rectangle ($(an)+(.3, -.3)$);	
            \node at (\dis, 10.9) {$V_H^+$};
            \node[tiny vertex] (a1') at (\dis, 10) {};
            \node[tiny vertex] (a2') at (\dis, 9) {};
            \node[tiny vertex] (a3') at (\dis, 8) {};
            \node[tiny vertex] (a4') at (\dis, 7) {};
            \node at (\dis, 6) {$\vdots$};
            \node[tiny vertex] (an') at (\dis, 5) {};
            \draw[box] ($(a1')+(-.3, .3)$) rectangle ($(an')+(.3, -.3)$);
            \node at (\dis+1.5,7.5) {$\longrightarrow$};
            \draw (a1) -- (a2');
            \draw (a1) -- (a4');
            \draw (a2) -- (a3');
            \draw (a2) -- (a4');
            \draw (a2) -- (an');
            \draw (a3) -- (a2');
            \draw (a3) -- (a3');
            \draw (a3) -- (a4');
            \draw (a4) -- (a2');
            \draw (a4) -- (an');
            \draw (an) -- (a1');
            \draw (an) -- (a3');
        \end{scope}
        \begin{scope}[xshift=10.75cm]    
            \path[rounded corners,fill=gray!15] (-.7, 11.3) rectangle (\dis.7, 4.5);

            \foreach \i in {1}{ 
                \node[tiny vertex] (c\i) at (-2, 11-\i) {};
                \node at ($(c\i)+(-0.8, 0)$)  {$v^+_{\i}$};
            }
            \foreach \i in {1.7}{ 
                \node[tiny vertex] (c\i) at (-2, 11-\i) {};
                \node at ($(-2.8, 9.3)$)  {$v^+_2$};
            }
            \node at (-2, 8.8) {$\vdots$};
            \foreach \i in {3}{ 
                \node[tiny vertex] (c\i) at (-2, 11-\i) {};
                \node at ($(c\i)+(-0.8, 0)$) {$v^+_{n^-}$};
            }
            \draw[box] ($(-2,10)+(-.3, .3)$) rectangle ($(-2, 8)+(.3, -.3)$);
            \draw[curly,leadercolor] ($(-2,8.6)+(.3, 0)$) -- (-.3, 8.6);
            
            \foreach \i in {1}{ 
                \node[tiny vertex] (d\i) at (\dis +2, 11-\i) {};
                \node at ($(d\i)+(0.8, 0)$)  {$v^-_{\i}$};
            }
            \foreach \i in {1.7}{ 
                \node[tiny vertex] (d\i) at (\dis +2, 11-\i) {};
                \node at ($(\dis+1.4+1.4, 9.3)$)  {$v^-_2$};
            }
            \node at (\dis+2, 8.8) {$\vdots$};
            \foreach \i in {3}{ 
                \node[tiny vertex] (d\i) at (\dis +2, 11-\i) {};
                \node at ($(d\i)+(0.8, 0)$)  {$v^-_{n^+}$};
            }
            \draw[box] ($(\dis+2,10)+(-.3, .3)$) rectangle ($(\dis+2, 8)+(.3, -.3)$);
            \draw[curly,leadercolor] ($(\dis+2,8.6)+(-.3, 0)$) -- (\dis+.3, 8.6);

            \node at (0, 10.9) {$V_H^-$};
            \node[tiny vertex] (a1) at (0, 10) {};
            \node[tiny vertex] (a2) at (0, 9) {};
            \node[tiny vertex] (a3) at (0, 8) {};
            \node[tiny vertex] (a4) at (0, 7) {};
            \node at (0, 6) {$\vdots$};
            \node[tiny vertex] (an) at (0, 5) {};
            \draw[box] ($(a1)+(-.3, .3)$) rectangle ($(an)+(.3, -.3)$);	
            \node at (\dis, 10.9) {$V_H^+$};
            \node[tiny vertex] (a1') at (\dis, 10) {};
            \node[tiny vertex] (a2') at (\dis, 9) {};
            \node[tiny vertex] (a3') at (\dis, 8) {};
            \node[tiny vertex] (a4') at (\dis, 7) {};
            \node at (\dis, 6) {$\vdots$};
            \node[tiny vertex] (an') at (\dis, 5) {};
            \draw[box] ($(a1')+(-.3, .3)$) rectangle ($(an')+(.3, -.3)$);
            \draw[followercolor] (a1) -- (a1');
            \draw[followercolor] (a1) -- (an');
            \draw[followercolor] (a1) -- (a3');
            \draw[followercolor] (a2) -- (a2');
            \draw[followercolor] (a2) -- (a1');
            \draw[followercolor] (a3) -- (an');
            \draw[followercolor] (a3) -- (a1');
            \draw[followercolor] (a4) -- (a4');
            \draw[followercolor] (a4) -- (a1');
            \draw[followercolor] (a4) -- (a3');
            \draw[followercolor] (an) -- (a2');
            \draw[followercolor] (an) -- (a4');
            \draw[followercolor] (an) -- (an');
        \end{scope}
    \end{tikzpicture}
	\end{center}
	\caption{An illustration of our polynomial transformation from \bcbs to \rp{\maxBM}. 
	A straight line denotes a single edge, while a wavy line represents a complete bipartite graph between the corresponding vertex sets.
    The edge sets $I_\ell$ and $I_f$ are depicted in blue and red, respectively.}
	\label{fig:response-max-bipartitematching}
\end{figure}

Hence, complete \rp{\maxM} is \np-hard on general graphs.
In particular, the complexity of \rp{\maxM} strongly depends on the type of edges, \ie whether they are mandatory or forbidden. In contrast, we next show that the complete response problem for \emph{perfect} matchings is always tractable. For a set~$I'$ of edges of~$G$, we denote by~$G\setminus V(I')$ the graph resulting from~$G$ when removing all endpoints of edges in~$I'$, together with their incident edges. 
If~$I'$ is a matching of~$G$, then~$G\setminus V(I')$ is just the part of~$G$ not covered by~$I'$. 

\begin{theorem}\label{thm:perfect-no-indifferent-follower}
    If~$I_f=I_0\cup I_1$, then \rp{\perfectM} is polynomial-time solvable.
\end{theorem}

\begin{proof}
    If each follower's edge is either forbidden or mandatory, it is clear for every vertex whether the leader or the follower has to cover it for obtaining a perfect matching.
    The problem reduces to determining the existence of two minimum-weight perfect matchings in two disjoint graphs:
    First, the leader has to exactly cover all vertices that are not incident to mandatory edges, \ie we need a perfect matching $\LL$ on the leader's edges in the graph~$G\setminus V(I_1)$.
    Second, we need to verify that the follower chooses $I_1$, \ie that $I_1$ is an optimal perfect matching on the follower's edges in the remaining graph $G\setminus V(\LL)$ \wrt $d$.
\end{proof}

\subsection{Partial Response Setting}\label{subsec:partial}
For partial responses, both \rp{\maxBM} and \rp{\perfectBM} are \np-hard, 
even with the slightest restriction of a single forbidden or mandatory edge.
We have already seen in~\Cref{ex:knapsack} that neutral edges can make a response problem significantly harder.

\begin{theorem}\label{thm:np-hard:I_*arbitrary:I_1=1}
    \rp{\maxBM} is \np-hard even when $|I_1|=1$ and~$|I_0|=0$.
\end{theorem}
\begin{proof}
    For a \bcbs instance $\instance = (H, k)$ with $H=(V_H, E_H)$ and a bipartition $V_H~=~V_H^-\ \dot\cup\ V_H^+$, we construct an instance of \rp{\maxBM} similarly to the one for~\Cref{thm:problembm-np-hard}:
    We define the bipartite graph~$G=(V^-\ \dot\cup\ V^+,E)$ by
    \[
    V^-\coloneqq  V_H^-\cup\{v^-_1,\dots,v^-_{n^+}\}\cup\{p_1,\dots,p_k\}\cup\{z'\},\quad V^+\coloneqq  V_H^+\cup\{v^+_1,\dots,v^+_{n^-}\}\cup\{z\}
    \]
    and edges~$E=I_\ell\cup I_f$ with
    \begin{eqnarray*}
    I_\ell & \coloneqq & \{(u,v^+_i)\mid u\in V_H^-,i\in\numset{n^-}\}\cup\{(v^-_i,u)\mid i\in\numset{n^+},u\in V_H^+\},\\
    I_f & \coloneqq & \{(u,v)\mid u\in V_H^-,v\in V_H^+, (u,v)\not\in E_H\} \cup \{(p_j,v)\mid j \in \numset{k}, v \in V_H^+\}\\
    & & \cup\, \{(v^-_i,z)\mid i \in \numset{n^+}\} \cup \{(p_j,z)\mid j \in \numset{k}\}\cup\{(z, z')\}.
    \end{eqnarray*}
    Let $I_1 \coloneqq  \set{(z,z')}$, $I_*\coloneqq I_f \setminus I_1$, and $I_0\coloneqq \emptyset$.
    We define the follower's weight function $d\colon I_f\to \Q$ as
    \[
        d(e)\coloneqq \begin{cases}\begin{array}{rl}
            3 & \text{ if } e=(u,v) \text{ with } u\in V_H^-, v\in V_H^+, \\
            1 & \text{ if } e=(z,z'), \\
            2 & \text{ otherwise}.
            \end{array}
        \end{cases}
    \]
    This construction can be done in polynomial time; see~\Cref{fig:np-hard:I_0=1:indif:arbitrary} for an illustration.
    
    \begin{figure}[t]
		\begin{center}
		\begin{tikzpicture}[yscale=.9,scale=0.8]
            \def\dis{5}
            \begin{scope}
                \node at (0, 10.9) {$V_H^-$};
                \node[tiny vertex] (a1) at (0, 10) {};
                \node[tiny vertex] (a2) at (0, 9) {};
                \node[tiny vertex] (a3) at (0, 8) {};
                \node[tiny vertex] (a4) at (0, 7) {};
                \node at (0, 6) {$\vdots$};
                \node[tiny vertex] (an) at (0, 5) {};
                \draw[box] ($(a1)+(-.3, .3)$) rectangle ($(an)+(.3, -.3)$);	
                \node at (\dis, 10.9) {$V_H^+$};
                \node[tiny vertex] (a1') at (\dis, 10) {};
                \node[tiny vertex] (a2') at (\dis, 9) {};
                \node[tiny vertex] (a3') at (\dis, 8) {};
                \node[tiny vertex] (a4') at (\dis, 7) {};
                \node at (\dis, 6) {$\vdots$};
                \node[tiny vertex] (an') at (\dis, 5) {};
                \draw[box] ($(a1')+(-.3, .3)$) rectangle ($(an')+(.3, -.3)$);
                \node at (\dis+1,7.5) {$\longrightarrow$};
                \draw (a1) -- (a2');
                \draw (a1) -- (a4');
                \draw (a2) -- (a3');
                \draw (a2) -- (a4');
                \draw (a2) -- (an');
                \draw (a3) -- (a2');
                \draw (a3) -- (a3');
                \draw (a3) -- (a4');
                \draw (a4) -- (a2');
                \draw (a4) -- (an');
                \draw (an) -- (a1');
                \draw (an) -- (a3');
            \end{scope}
            \begin{scope}[xshift=9.5cm]    
                \path[rounded corners,fill=gray!15] (-.7, 11.3) rectangle (\dis.7, 4.5);

                \foreach \i in {1}{ 
                    \node[tiny vertex] (c\i) at (-2, 11-\i) {};
                    \node at ($(c\i)+(.7, 0.0)$)  {$v^+_{\i}$};
                }
                \foreach \i in {1.7}{ 
                    \node[tiny vertex] (c\i) at (-2, 11-\i) {};
                    \node at ($(-1.3, 9.3)$)  {$v^+_2$};
                }
                \node at (-2, 8.8) {$\vdots$};
                \foreach \i in {3}{ 
                    \node[tiny vertex] (c\i) at (-2, 11-\i) {};
                    \node at ($(c\i)+(0.82, 0)$) {$v^+_{n^-}$};
                }
                \draw[box] ($(-2,10)+(-.3, .3)$) rectangle ($(-2, 8)+(.3, -.3)$);
                \draw[curly,leadercolor] ($(-2,8.6)+(.3, 0)$) -- (-.3, 8.6);
                
                \foreach \i in {1}{ 
                    \node[tiny vertex] (d\i) at (\dis +2, 11-\i) {};
                    \node at ($(d\i)+(-.7, 0)$)  {$v^-_{\i}$};
                }
                \foreach \i in {1.7}{ 
                    \node[tiny vertex] (d\i) at (\dis +2, 11-\i) {};
                    \node at ($(\dis+1.3, 9.3)$)  {$v^-_2$};
                }
                \node at (\dis+2, 8.8) {$\vdots$};
                \foreach \i in {3}{ 
                    \node[tiny vertex] (d\i) at (\dis +2, 11-\i) {};
                    \node at ($(d\i)+(-.75, 0)$)  {$v^-_{n^+}$};
                }
                \draw[box] ($(\dis+2,10)+(-.3, .3)$) rectangle ($(\dis+2, 8)+(.3, -.3)$);
                \draw[curly,leadercolor] ($(\dis+2,8.6)+(-.3, 0)$) -- (\dis+.3, 8.6);
                
                \node at (\dis+2, 5.7) {$\vdots$};
                \foreach \i in {4}{ 
                    \node[tiny vertex] (d\i) at (\dis +2, 11-\i) {};
                    \node at ($(d\i)-(.65, 0)$)  {$p_1$};
                }
                \foreach \i in {4.7}{ 
                    \node[tiny vertex] (d\i) at (\dis +2, 11-\i) {};
                    \node at (\dis +1.3, 11-\i)  {$p_2$};
                }
                \foreach \i in {6}{ 
                    \node[tiny vertex] (d\i) at (\dis +2, 11-\i) {};
                    \node at ($(d\i)-(.65, 0)$)  {$p_k$};
                }
                \draw[box] (\dis+1.7, 7.3) rectangle (\dis+2.3,4.7);
                \draw[curly,followercolor] (\dis+1.7, 5.7) -- (\dis+.3, 5.7);
                
                \node[tiny vertex] (z) at (\dis+3.5, 7.5) {};
                \node at ($(z)-(0, .6)$)  {$z$};
                \draw[box] ($(z)+(-.3, .3)$) rectangle ($(z)+(.3, -.3)$);
                \draw[curly,followercolor] (\dis+2.3, 6) -- (\dis+3.2, 7.3);
                \node[tiny vertex] (z') at (\dis +4.5, 7.5) {};
                \node at ($(z')-(0, .5)$)  {$z'$};
                \draw[followercolor,line width=2.5pt] (z) -- (z');
                \draw[curly,followercolor] (\dis+2.3, 9) -- (\dis+3.2, 7.7);

                \node at (0, 10.9) {$V_H^-$};
                \node[tiny vertex] (a1) at (0, 10) {};
                \node[tiny vertex] (a2) at (0, 9) {};
                \node[tiny vertex] (a3) at (0, 8) {};
                \node[tiny vertex] (a4) at (0, 7) {};
                \node at (0, 6) {$\vdots$};
                \node[tiny vertex] (an) at (0, 5) {};
                \draw[box] ($(a1)+(-.3, .3)$) rectangle ($(an)+(.3, -.3)$);	
                \node at (\dis, 10.9) {$V_H^+$};
                \node[tiny vertex] (a1') at (\dis, 10) {};
                \node[tiny vertex] (a2') at (\dis, 9) {};
                \node[tiny vertex] (a3') at (\dis, 8) {};
                \node[tiny vertex] (a4') at (\dis, 7) {};
                \node at (\dis, 6) {$\vdots$};
                \node[tiny vertex] (an') at (\dis, 5) {};
                \draw[box] ($(a1')+(-.3, .3)$) rectangle ($(an')+(.3, -.3)$);
                \draw[followercolor] (a1) -- (a1');
                \draw[followercolor] (a1) -- (an');
                \draw[followercolor] (a1) -- (a3');
                \draw[followercolor] (a2) -- (a2');
                \draw[followercolor] (a2) -- (a1');
                \draw[followercolor] (a3) -- (an');
                \draw[followercolor] (a3) -- (a1');
                \draw[followercolor] (a4) -- (a4');
                \draw[followercolor] (a4) -- (a1');
                \draw[followercolor] (a4) -- (a3');
                \draw[followercolor] (an) -- (a2');
                \draw[followercolor] (an) -- (a4');
                \draw[followercolor] (an) -- (an');
            \end{scope}
        \end{tikzpicture}
		\end{center}
		\caption{Polynomial transformation from \bcbs to \rp{\maxBM} for the case where $|I_1|=1$. 
		A straight line denotes a single edge, while a wavy line represents a complete bipartite graph between the corresponding vertex sets.
        The edge sets $I_\ell$ and $I_f$ are depicted in blue and red, respectively; the bold edge is mandatory.
        }
		\label{fig:np-hard:I_0=1:indif:arbitrary}
	\end{figure}
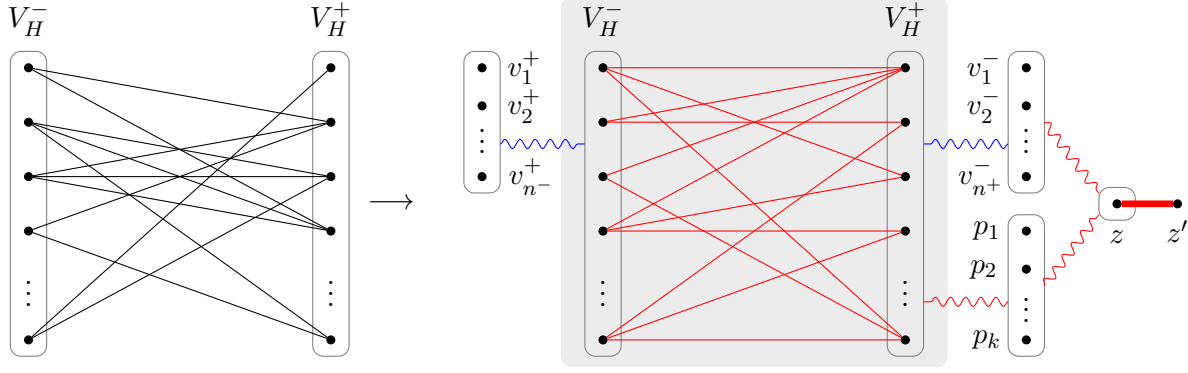
   
    Assume $\instance$ is a yes-instance of \bcbs, then there exists a $(k\times k)$-complete bipartite subgraph of~$H$.
    Let $\LL \subseteq I_\ell$ be a matching that covers all remaining vertices in~$H$ that are not part of the complete bipartite subgraph.
    Then at least one endpoint of each edge between~$V_H^-$ and~$V_H^+$ in~$G$ is covered by~$\LL$.
    The remaining graph~$G \setminus V(\LL)$ contains a complete bipartite graph consisting of exactly~$k$ vertices from~$V_H^+$ and~$z$ on one side and~$p_1, \dots, p_k$ on the other side, as well as the edge $(z, z')$ and $k$ isolated vertices from~$V_H^-$.
    All its edges except the mandatory edge~$(z, z')$ have the same weight for the follower.
    The optimal follower's choice is thus to choose~$k$ edges between $V_H^+$ and $\{p_1, \dots, p_k\}$ and the mandatory edge.
    Any matching including an edge $(p_j, z)$ has strictly smaller weight as it can only include at most $k$ edges.
    Thus, the follower is guaranteed to choose~$(z,z')$.

    For the converse, assume that there exists a leader's solution~$\LL$ such that the follower's response~$\FF$ contains the mandatory edge~$(z,z')$.
    Then, all vertices~$v^-_1,\dots, v^-_{n^+}$ and~$p_1,\dots, p_k$ must be covered by some edge in~$\LL\cup\FF$ since otherwise the follower would have preferred any of his incident edges over the edge~$(z,z')$.
    Thus,~$n^+$ of the vertices in~$V_H^+$ must be covered by some edges in~$\LL$ and the remaining $k$ vertices in~$V_H^+$ must be covered by edges incident to vertices~$p_1,\dots, p_k$.
    The latter is only possible if the remaining~$k$ vertices in~$V_H^+$ are isolated from the non-covered vertices in~$V_H^-$, since otherwise the follower would always have preferred to select an available edge~$(u,v)$, where $u\in V_H^-, v\in V_H^+$, over the edge~$(v,p_i)$ for some $i\in\numset{k}$.
    Additionally, since at most~$n^-$ vertices in~$V_H^-$ can be covered by the leader, this implies that the subgraph of~$V_H$ in~$G$ that is not covered by~$\LL$ forms an independent set with at least~$k$ vertices on each side. 
    So~$\instance$ is a yes-instance.
\end{proof}

Similar to the above result, we can obtain \np-hardness for the case when~$|I_0|=1$, $|I_1|=0$, and~$|I_*|$ is arbitrary. 
For 
this, we can add a vertex $z''$ as a pendant to $z'$ in the above construction, and let $d(z',z'')=1$.
Finally, redefine $I_0\coloneqq \{(z',z'')\}$, $I_*\coloneqq I_f \setminus \{(z',z'')\}$, and $I_1\coloneqq \emptyset$.

\begin{corollary}\label[corollary]{cor:np-hard:I_*arbitrary:I_0=1}
    \rp{\maxBM} is \np-hard even when $|I_0|=1$ and~$|I_1|=0$.
\end{corollary}

The equivalence of the single-level problems \maxM\ and \perfectM\ can be shown by transformations of instances of the one into the other problem; see \eg \cite[Proposition~11.1]{KorteVygen_CombOpt}.
Similarly, a transformation is also possible from~\rp{\maxM} to~\rp{\perfectM}.
Given an instance of~\rp{\maxM} with a graph~$G=(V,E)$, we add a copy for each vertex in~$V$ and connect it to all other vertices of the graph by zero-weight edges.
All added edges become neutral follower's edges, \ie are added to~$I_*$.
Now for any leader's solution~$\LL\subseteq I_\ell$, it is optimal for the follower to choose the same response~$\FF$ as in the initial instance of~\rp{\maxM} and extend it to a perfect matching by a suitable selection of the new auxiliary edges.

For a transformation from~\rp{\maxBM} to~\rp{\perfectBM}, instead of connecting the added copies to all vertices of the graph, we add the copy to the opposite partite set of the original vertex and connect it to all vertices of the partite set of the original vertex, thereby preserving the bipartition.
It will still be possible for the follower to extend~$\LL\cup\FF$ by auxiliary edges to a perfect matching since with~$\LL\cup\FF$ more copied vertices than uncovered original vertices remain.
Note that these transformations only extend the set of neutral edges, exactly preserving the sets of mandatory and forbidden edges.
Due to the total number of neutral edges being already arbitrary in~\Cref{thm:np-hard:I_*arbitrary:I_1=1} and~\Cref{cor:np-hard:I_*arbitrary:I_0=1}, the \np-hardness directly extends to the perfect matching case.

\begin{corollary}\label[corollary]{cor:perfect-np-hard:I_*arbitrary:I_1orI_0=1}
    \rp{\perfectBM} is \np-hard even when~$|I_0\cup I_1|=1$.
\end{corollary}

\subsection{Parameterized Results}
\label{subsec:parameterzied}

Next, we investigate the parameterized complexity of \rp{\maxM} and \rp{\perfectM} \wrt the cardinalities of the sets~$I_\ell$, $I_0$, $I_1$, and $I_*$. 
The results of~\Cref{subsec:partial} imply that \np-hardness persists even when~$|I_0|$ and~$|I_1|$ are fixed. 
In contrast, both problems are fixed-parameter tractable~(\fpt) \wrt both~$|I_*|$ and~$|I_\ell|$; see \Cref{pro:polytime:I_f=I_*} and \Cref{cor:ftp-rp}.

The fixed-parameter tractability in terms of the different types of follower's edges turns out to be more interesting.
Complete \rp{\maxM} is \fpt \wrt the total number of forbidden edges, even when the total number of mandatory edges is unbounded.
This confirms the previous observation that mandatory edges tend to be easier to deal with than forbidden edges, compare~\Cref{prop:polytime:I_f=I_1} and~\Cref{thm:problembm-np-hard}.

\begin{theorem}\label{lem:fpt:completeMaxM:I_0}
    If $I_f=I_0\cup I_1$, then \rp{\maxM} is \fpt \wrt $|I_0|$.
\end{theorem}
\begin{proof}
    To obtain the desired response, the leader must cover an appropriate subset $S\subseteq V(I_0) \setminus V(I_1)$ of vertices incident to forbidden edges such that the follower's response consists of exactly~$I_1$.
    If~$|I_0|=k$, we can enumerate over all such subsets $S$, which are at most $2^{2k}$ many.
    For a subset~$S$, it suffices to verify whether the leader can choose a matching~$\LL \subseteq I_\ell$ that covers exactly~$S$ within~$V(I_f)$ and whether the follower's reaction renders the set~$I_1$.
    We claim that the former holds if and only if the maximum-weight matching in the graph $G\setminus(V(I_f)\setminus S)$ \wrt 
    \[
    c(e)\coloneqq\begin{cases}\begin{array}{ll}
        1 & \text{if } e=(u,v) \text{ with } u\in S,\ v\notin S,\\
        2 & \text{if } e=(u,v) \text{ with } u, v\in S,\\
        0 & \text{otherwise},
        \end{array}
    \end{cases}
    \]
    has weight $|S|$.
    The weight of any matching can be at most $|S|$, since each vertex in~$S$ is adjacent to at most one edge of the matching. 
    Moreover, if the optimal weight is strictly less than~$|S|$, some vertex in $S$ is not covered by the matching. 
    Additionally, to check the follower's response, it suffices to compute an optimal matching on the follower's edges in 
    $G\setminus S$ \wrt~$d$.
\end{proof}

When $|I_0|+|I_*|\leq k$, there are $2^{|I_*|} \leq 2^k$ many possibilities 
to turn neutral edges into mandatory of forbidden edges.
Each resulting 
complete response problem has at most $k$ forbidden edges, so applying~\Cref{lem:fpt:completeMaxM:I_0} yields the following corollary.

\begin{corollary}\label[corollary]{thm:fpt:I_0:I_*}
    \rp{\maxM} is \fpt with respect to $|I_0|+|I_*|$.
\end{corollary}

With a similar argument, for \rp{\perfectM} with $|I_*| = k$, there are $2^k$ corresponding complete response problems that can be solved in polynomial time by~\Cref{thm:perfect-no-indifferent-follower}.

\begin{corollary}\label[corollary]{pro:fpt:perfectMconstIndifferent}
    \rp{\perfectM} is \fpt with respect to $|I_*|$.
\end{corollary}

\Cref{thm:fpt:I_0:I_*} and~\Cref{pro:fpt:perfectMconstIndifferent} imply that \rp{\maxM} and \rp{\perfectM} are both \fpt in the total number of follower's edges.
Interestingly, in~\cite[Theorem~20]{BHH22_ComplexityBilevelSpanningTree} the same result has been shown for the spanning tree response problem.

\section{Conclusion}
\label{Sec:Conclusions}

In this study, we settle the complexity of the response problem for both maximum-weight and perfect matching problems for all the main cases \wrt restrictions on the total number of leader's, mandatory, forbidden, and neutral edges; see \cref{tab:overview} for an overview of our results.

\begin{table}[htb]
    \centering
    \scalebox{1}{
    \begin{threeparttable}
        \begin{tabular}{cccc l@{\hspace{1cm}} rl@{\hspace{1cm}} rl} \toprule
            & \multicolumn{3}{c}{{\mathversion{bold}$\abs{I_f}$}}  && \multicolumn{2}{l}{\multirow[b]{2}{*}{\makecell[l]{\bf \rp{\maxM} \\ \bf \rp{\maxBM}}}}    & \multicolumn{2}{l}{\multirow[b]{2}{*}{\makecell[l]{\bf \rp{\perfectM} \\ \bf \rp{\perfectBM}}}}  \\ \cmidrule[0.5pt]{2-4}
            {\mathversion{bold}$\abs{I_\ell}$}  & {\mathversion{bold}$\abs{I_*}$}   & {\mathversion{bold}$\abs{I_0}$}   & {\mathversion{bold}$\abs{I_1}$}   &   &  & &  \\ 
            \midrule 
            \multicolumn{8}{l}{Bounded leader set} \\[1ex]
            $k$             & {}            & {}            & {}            && \fpt &Cor.~\ref{cor:ftp-rp}    & \fpt &Cor.~\ref{cor:ftp-rp}\\ 
            \midrule 
            \multicolumn{8}{l}{Complete response} \\[1ex]
            {}              & 0             & 0             & {}            &&  \p & Prop.~\ref{prop:polytime:I_f=I_1}         & \p  & \multirow{4}{*}{$\left.\vphantom{\begin{matrix}a\\a\\a\\a\end{matrix}}\right]$~Thm.~\ref{thm:perfect-no-indifferent-follower}} \\
            {}              & 0             & {}            & 0             && \np-hard &\multirow{2}{*}{$\left.\vphantom{\begin{matrix}a\\a\end{matrix}}\right]$~Thm.~\ref{thm:problembm-np-hard}}           & \p & \\
            {}              & 0             & {}            & {}            && \np-hard &      & \p & \\
            {}              & 0             & $k$           & {}            && \fpt &Thm.~\ref{lem:fpt:completeMaxM:I_0}            & \p & \\
            \midrule 
            \multicolumn{8}{l}{\makecell[l]{Partial response with bounded~$|I_*|$}} \\[1ex]
            {}              & $k$           & {}           & {}            && \np-hard & Thm.~\ref{thm:problembm-np-hard}                    &  \fpt & \multirow{2}{*}{$\left.\vphantom{\begin{matrix}a\\a\end{matrix}}\right]$~Cor.~\ref{pro:fpt:perfectMconstIndifferent}} \\
            {}              & \multicolumn{2}{c}{$\leftarrow k\rightarrow$}            & {}            && \fpt &Cor.~\ref{thm:fpt:I_0:I_*}      & \fpt &  \\ \midrule 
            \multicolumn{8}{l}{Partial response} \\[1ex]
            {}              & {}            & 0             & 0             && \p &Prop.~\ref{pro:polytime:I_f=I_*}                  & \p &Prop.~\ref{pro:polytime:I_f=I_*}\\
            {}              & {}            & 0             & 1             && \np-hard &Thm.~\ref{thm:np-hard:I_*arbitrary:I_1=1}  & \np-hard & \multirow{2}{*}{$\left.\vphantom{\begin{matrix}a\\a\end{matrix}}\right]$~Cor.~\ref{cor:perfect-np-hard:I_*arbitrary:I_1orI_0=1}} \\
            {}              & {}            & 1             & 0             && \np-hard &Cor.~\ref{cor:np-hard:I_*arbitrary:I_0=1}  & \np-hard & \\ \bottomrule
        \end{tabular}
        \caption{Overview of our results on the matching response problem depending on the sizes of sets $I_\ell, I_*, I_0$, and $I_1$ with $k$ being a constant. 
        Empty cells imply that the results hold for arbitrary values.
        \fpt means fixed-parameter tractable in~$k$.}
    \label{tab:overview}
    \end{threeparttable}}
\end{table}

Though maximum-weight and perfect matching problems are equivalent as single-level problems, their response problems differ in complexity in several cases.
Interestingly, however, we did not encounter any situation in which restricting to bipartite graphs made the response problem easier to solve in either case.

Our tractability results for matching response problems directly extend to the corresponding optimization problems, where the leader optimizes her objective instead of finding a feasible solution, while keeping the notion of mandatory, forbidden, and neutral edges.
This is not true for all underlying problems, as can be seen from the case 
$I_f=I_1$ in \cref{ex:knapsack}.
Moreover, in a response problem, the leader does not necessarily have to solve the follower's subproblem in order to decide whether a response from the target set can be achieved, suggesting that in certain cases the problem might be easier than the corresponding partitioned-items problem or even the follower's problem.

For future work, one might investigate whether the general response problem \rp{\problem}, \ie without restrictions on $I_\ell$, $I_0$, $I_1$, and $I_*$, for an underlying problem \problem can be complete for a lower level of the polynomial hierarchy than its corresponding partitioned-items bilevel variant~\bil{\problem}.
Moreover, considering the rich complexity landscape we encountered for matching response problems, we suggest to study response problems for further optimization problems, \eg shortest path, knapsack, or independent set.

\clearpage
\appendix
\section{Appendix}

\begin{proposition} \label[proposition]{lem:trans-bil-opt-pes}
    Given a combinatorial optimization problem~$\problem$, any optimistic (resp.\ pessimistic) instance of \bil{\problem} can be polynomially transformed into an equivalent pessimistic (resp.\ optimistic) instance of~$\bil{\problem}$ by only transforming the follower's objective function~$d$.
    For integer-valued~$d\colon I_f\to\Z$, feasible transformations for the follower's objective function are:
    \begin{itemize}
        \item[--] optimistic to pessimistic:\quad $\tilde d \coloneqq d + \epsilon\cdot \left.c\right|_{I_f}$\\[-2.75ex]
        \item[--] pessimistic to optimistic:\quad $\tilde d \coloneqq d - \epsilon\cdot \left.c\right|_{I_f}$
    \end{itemize}
    with~$\epsilon=0$ if~$\left.c\right|_{I_f}\equiv 0$ and~$\varepsilon < \big(\sum_{i\in I_f}|c(i)|\big)^{-1}$ otherwise.
\end{proposition}

\begin{proof}
    If~$\left.c\right|_{I_f}\equiv 0$, then the leader is indifferent about the follower's response, so the optimistic and pessimistic setting are equivalent.
    
    In the following, let~$\feas(\LL)\coloneqq \{ \FF\subseteq I_f \mid \LL\cup \FF\in\feas\}$ denote the set of all feasible follower's solutions and~$\feas^*_d(\LL)\coloneqq\argmax \set{d(\FF)\mid \FF \in \feas(\LL)}$ the set of all optimal follower's responses to a given leader's solution~$\LL\subseteq I_\ell$ for an underlying follower's weight function~$d$.
    We prove the correctness of the transformation from the optimistic to the pessimistic setting.
    The reverse transformation can be validated similarly.
    
    Let an instance of~\bil{\problem} with follower's weight function~$d\colon I_f\to\Z$ in the optimistic setting be given.
    We show that the transformation preserves the leader's objective function value for every leader's solution~$\LL\subseteq I_\ell$.
    First, we prove that the transformed set of optimal follower's responses~$\feas^*_{\tilde d}(\LL)$ contains exactly the initially optimistic optimal follower's responses,~\ie the set~$\FF\in\feas^*_d(\LL)$ that maximize~$c$.
    By integrality of~$d$, we have~$d(\FF^*) \geq d(\FF) + 1$ for all optimal follower's responses~$\FF^*~\in~\feas^*_d(\LL)$ and all non-optimal follower's solutions~$\FF \in \feas(\LL) \setminus \feas^*_d(\LL)$.
    Furthermore, by the choice of~$\varepsilon$, we have~$\tilde d(\FF) < d(\FF^*) \leq \tilde d(\FF^*)$.
    So any initially non-optimal follower's solution~$\FF$ to~$\LL$ remains non-optimal after the transformation,~\ie~$\feas^*_{\tilde d}(\LL)~\subseteq~\feas^*_{d}(\LL)$.
    
    Now let~$\FF^*\in \feas^*_d(\LL)$ be an optimistic optimal follower's response to~$\LL$, \ie 
    \[\FF^*\in\arg\max\{c(\FF) \mid \FF\in\feas^*_d(\LL)\}\;.\]
    Thus, any initially optimal follower's response that was not an optimistic solution,~\ie a~$\FF'\in\feas^*_d(\LL)$ with~$c(\FF') < c(\FF^*)$, does not remain follower-optimal after the transformation since
    \[\tilde d(\FF') = d(\FF')+\epsilon\cdot c(\FF') < \tilde d(\FF^*)\;.\]
    Therefore,~$\feas^*_{\tilde d}(\LL) \subseteq \arg\max\{c(\FF) \mid \FF\in\feas^*_d(\LL)\}$.
    Moreover, since~$\tilde d(\FF^*) = d(\FF^*)+\epsilon\cdot c(\FF^*)$ is equal for all~$\FF^*\in\arg\max\{c(\FF) \mid \FF\in\feas^*_d(\LL)\}$, it holds~$\feas^*_{\tilde d}(\LL) = \arg\max\{c(\FF) \mid \FF\in\feas^*_d(\LL)\}$ which we wanted to prove.
    But this implies that the leader's objective function value for~$\LL$ in the transformed pessimistic instance is equal to the one in the initial optimistic instance, since
    \begin{align*}
        \min\{c(\FF) \mid \FF\in\feas^*_{\tilde d}(\LL)\}
        &= \min\big\{c(\FF) \mid \FF\in\arg\max\{c(\FF) \mid \FF\in\feas^*_d(\LL)\}\big\} \\
        &= \max\{c(\FF) \mid \FF\in\feas^*_d(\LL)\}.
    \end{align*}
    Thus, the transformation preserves the leader's objective function value of every~$\LL\subseteq I_\ell$, proving that the transformed instances are equivalent.
\end{proof}
Note that the transformations described in \Cref{lem:trans-bil-opt-pes} are inverse to one another. The integrality assumption of~$d$ 
was only made to simplify the proof; it can always be obtained by scaling if~$d$ has rational values.
We next show that~$\bil{\problem}$ can be at most one level harder than~$\problem$ in the polynomial hierarchy.

\begin{theorem}\label{lem:complexity-bil}
    For any combinatorial optimization problem~$\problem$, the decision version of $\bil{\problem}$ belongs to~$\np^\problem$.
\end{theorem}

\begin{proof}
    We prove the statement for~$\bil{\problem}$, \ie the partitioned-items problem in the optimistic setting.
    For the pessimistic setting, the same result follows by first transforming a pessimistic into an optimistic instance, according to \cref{lem:trans-bil-opt-pes}.

    Consider an instance~$\mathcal J$ of the decision version of~\bil{\problem} based on an instance $\instance=(I, \feas, c)$ of the combinatorial optimization problem \problem, along with a partition $I=I_\ell~\dot\cup~I_f$, a follower's objective function $d$, and a threshold $k\in \Q$.
    The instance~$\mathcal J$ is a yes-instance if there exists a leader's solution~$\LL\subseteq I_\ell$ and a follower's response~$\FF\in\arg\max\{d(\FF') \mid \FF'\subseteq I_f,\ \LL\cup \FF'\in\feas\}$ such that~$c(\LL \cup \FF)\geq k$.
    We show that such an~$\LL$ is a certificate for a yes-instances of~$\mathcal J$ that can be verified in polynomial time using an oracle for~\problem.
    To this end, we define an instance~$\instance_\LL = (I,\feas,c_\LL)$ of~\problem with
    weight function $c_\LL\colon I\to\Q$ as follows:
                            
    \[
        c_\LL(i)\coloneqq \begin{cases}\begin{array}{cl}
            \phantom{+}\bigM & \text{ if } i\in \LL, \\
            -\bigM & \text{ if } i\in I_\ell\setminus\LL, \\
            d(i)+\epsilon\cdot c(i) & \text{ if } i\in I_f\;,
            \end{array}
        \end{cases}
    \]
    where~$\bigM,\epsilon\in\Q$ with~$\bigM > \max\{\sum_{i\in I_f} \abs{d(i)}, 1\}$ and~$\epsilon<\big(\sum_{i \in I_f}\abs{c(i)}\big)^{-1}$.
    For the well-definedness of~$\epsilon$, we may assume
    that there exists an~$i\in I_f$ with~$c(i)\neq 0$. Otherwise, the instance~$\cal J$ of~$\bil{\problem}$ is equivalent to the instance~$\instance$ of~$\problem$ and could be decided with an oracle for~\problem.
    
    It suffices to show that~$\mathcal J$ is a yes-instance if and only if there exists a certificate~$\LL\subseteq I_\ell$ such that the optimal solution~$S$ returned by the oracle on input $\instance_\LL$ satisfies~$I_\ell\cap S = \LL$ and~$c_\LL(S)\geq k$, because the latter can be checked in polynomial time given the oracle for~$\problem$.
    Indeed, for a yes-instance, there exists $\LL\subseteq I_\ell$ with  an optimal follower's response~$\FF^*$ to~$\LL$ such that~$c(\LL\cup\FF^*) \geq k$.
    By the choice of~$\bigM$, the solution~$S$ will always contain~$\LL$ and be disjoint to~$I_\ell\setminus\LL$ when~$\LL$ is extended by some~$\FF$ to a feasible solution~$\LL\cup\FF\in\feas$.
    Moreover, by the choice of~$\epsilon$, the set~$S\setminus\LL$ will always be an optimistic optimal follower's response to~$\LL$ in the instance~$\mathcal J$.
    Thus,~$c_\LL(S) = c(\LL\cup\FF^*) \geq k$.
    
    For the other direction, let~$\LL\subseteq I_\ell$ such that the optimal solution~$S$ returned by the oracle on input $\instance_\LL$ satisfies~$I_\ell\cap S = \LL$ and~$c_\LL(S)\geq k$.
    Since~$I_\ell\cap S = \LL$, the certificate~$\LL$ is in fact a feasible leader's solution to~$\cal J$ as it can be extended by~$\FF\coloneqq S\setminus \LL\subseteq I_f$ to a feasible solution~$\LL~\cup~\FF~\in~\feas$.
    Moreover, by the choice of~$\epsilon$, the set~$\FF$ must be an optimal follower's response to~$\LL$ in the instance~$\mathcal J$ of~$\bil{\problem}$.
    Since~$c(\LL\cup \FF) = c_\LL(S)\geq k$, it follows that~$\mathcal J$ is a yes-instance.
\end{proof}
Note that it is crucial in this proof that the given oracle for~$\problem$ accepts arbitrary linear objective functions, possibly including negative coefficients. The shortest path problem is \np-hard in this regime, so that \Cref{lem:complexity-bil} does not contradict the $\Sigma_2^\text{P}$-hardness of the partitioned-items shortest path problem shown in~\cite{HW25_ComplexityBilevelShortestPath}.

\bibliographystyle{plain} %alpha
\bibliography{literature}

\end{document}